\documentclass[a4paper,12pt,reqno]{amsart}
\usepackage{amsmath, amsfonts, amssymb, enumerate, enumitem}
\usepackage[hang,flushmargin]{footmisc} 
\usepackage{amsthm}
\usepackage{mathtools}
\usepackage{calc}
\usepackage{bm}
\usepackage{comment}
\usepackage{hyperref}
\makeatletter
   \providecommand\@dotsep{5}
 \makeatother

\newtheorem{theorem}{Theorem}[section]
\newtheorem{lemma}[theorem]{Lemma}

\newtheorem{corollary}[theorem]{Corollary}
\newtheorem{proposition}[theorem]{Proposition}
\newtheorem{question}[theorem]{Question}

\theoremstyle{definition}
\newtheorem{defn}[theorem]{Definition}
\newtheorem{remark}[theorem]{Remark}
\newtheorem{example}[theorem]{Example}
\def\E{\mathbb{E}}
\def\Z{\mathbb{Z}}
\def\R{\mathbb{R}}
\def\T{\mathbb{T}}

\def\N{\mathbb{N}}
\def\P{\mathbb{P}}

\def\Q{\mathbb{Q}}

\newcommand{\ud}{\,\mathrm{d}}

\DeclareMathOperator{\cB}{\mathcal{B}}

\DeclareMathOperator{\cP}{\mathcal{P}}
\DeclareMathOperator{\cQ}{\mathcal{Q}}

\DeclareMathOperator{\cX}{\mathcal{X}}

\newcommand{\Zmod}[1]{\Z_{#1}} 

\let\originalleft\left
\let\originalright\right
\renewcommand{\left}{\mathopen{}\mathclose\bgroup\originalleft}
\renewcommand{\right}{\aftergroup\egroup\originalright}

\renewcommand{\subset}{\subseteq}
\renewcommand{\supset}{\supseteq}

\begin{document}

\title[Towers for commuting endomorphisms]{Towers for commuting endomorphisms, and combinatorial applications}

\author{Artur Avila}
\address{CNRS, IMJ-PRG, UMR 7586, Univ Paris Diderot, Sorbonne Paris Cit\'e\newline
\indent Sorbonnes Universit\'es, UPMC  Univ  Paris  06\newline
\indent F-75013,  Paris,  France\newline
\indent \&\newline
\indent  IMPA,  Estrada  Dona  Castorina  110,  Rio  de Janeiro, Brasil
}

\author{Pablo Candela}
\address{Alfr\'ed R\'enyi Institute of Mathematics\newline
	\indent 13-15 Re\'altanoda utca\newline
	\indent 1056 Budapest, Hungary}

\date{}
\maketitle

\begin{abstract}
We give an elementary proof of a generalization of Rokhlin's lemma for commuting non-invertible measure-preserving transformations, and we present several  combinatorial applications.
\end{abstract}

\section{Introduction}
\noindent Throughout this paper we denote by $(X,\cX,\mu)$ a standard probability space, and $T$ denotes an \emph{endomorphism} on $X$, that is, a measure-preserving transformation $X\to X$. By an $n$\emph{-tower} (or tower of height $n$) for $T$ we mean a sequence $B,\;T^{-1} B,\ldots,\;T^{-(n-1)}B$ of pairwise-disjoint successive preimages of some measurable set $B\subset X$. By the \emph{measure} of such a tower we mean simply $\mu\Big(\bigcup_{j=0}^{n-1} T^{-j}B \Big)$.\\
\indent Towers play an important role in proofs of several central results in ergodic theory, especially by providing ways to approximate a given endomorphism by a periodic one. Originally these methods focused on invertible transformations (automorphisms). The main tool powering these methods is the following well-known result, which was stated explicitly for the first time\footnote{A simple proof can be given for an ergodic map $T$ using the so-called skyscrapers of Kakutani (see \cite{Petersen}) and his name is also often associated with the result.} by Rokhlin \cite{Rok1}, and which concerns any automorphism $T$ on $X$ that is \emph{aperiodic}, meaning that we have $\mu(\{x\in X:T^n x=x\})=0$ for every positive integer $n$.

\begin{theorem}[Rokhlin]\label{thm:Rokh}
Let $\epsilon > 0$ and let $n$ be a positive integer. Then for every aperiodic automorphism $T$ on an atomless standard probability space, there exists an $n$-tower for $T$ of measure at least $1-\epsilon$.
\end{theorem}

\noindent The literature related to this very useful result is rich\footnote{The result is often referred to as Rokhlin's lemma, but in light of its importance it can also be stated as a theorem; see \cite{K&W}.}; we refer the reader to \cite{Korn,Weiss} for more detailed expositions. In particular, the result has been generalized in several directions.\\
\indent In one central direction, the $\Z$-action generated by $T$ is replaced with other group actions. Let us mention here the generalization to $\Z^d$-actions proved by Conze \cite{Conze} and independently by Katznelson and Weiss \cite{K&W}, and let us refer again to \cite{Korn,Weiss} for information on other such extensions.\\
\indent Another direction, less covered in the literature, concerns non-invertible maps. This starts with the version of Theorem \ref{thm:Rokh} in which $T$ is just an endomorphism. Up to the early 2000s, this version was part of the folklore (some explicit mentions of the result outline ways to prove it by modifying some of the existing proofs of the versions for automorphisms; see for instance \cite{Korn}). The first publication containing a full proof of a version for endomorphisms seems to be \cite{H&S}.\\ 
\indent It is natural to ask then for an analogue, for non-invertible maps, of the extension of Theorem \ref{thm:Rokh} to $\Z^d$-actions. This analogue is also motivated by some applications that we describe below. The main result in this paper provides such an analogue, with an elementary proof.\\
\indent The statement of the result uses the following terminology. Let $\N,\N_0$ denote the set of positive integers and non-negative integers respectively. We consider a \emph{measure-preserving action of} $\N_0^d$ on $X$, that is a map $f: \N_0^d\times X\to X$ such that for each $n=\big(n(1),\dots,n(d)\big)\in \N_0^d$ the map $f_n:X\to X$, $x\mapsto f(n,x)$ is an endomorphism on $X$, with $f_0$ being the identity map, and such that for every $m,n\in \N_0^d$ and $x\in X$ we have $f_{m+n}(x)=f_m (f_n(x))$. Equivalently, $f(n,x)=T_1^{n(1)}\circ T_2^{n(2)}\circ\cdots\circ T_d^{n(d)}(x)$ where $T_1,\ldots,T_d$ are commuting endomorphisms on $X$.\\
\indent We say that the action $f$ is \emph{free} if for every distinct $k,\ell\in \N_0^d$ we have
\begin{equation}\label{eq:freecond}
\mu(\{x\in X: f_k(x)=f_\ell(x)\})=0.
\end{equation}
\noindent For $k,\ell \in \N_0^d$, we write $k<\ell$ (respectively $k\leq \ell$) if for every $j\in [d]=\{1,2,\ldots,d\}$ we have $k(j)<\ell(j)$ (resp. $k(j)\leq \ell(j)$). For $n \in \N^d$ and $B \in \cX$, we denote by $B_{(n)}$ the union of preimages $\bigcup_{0\leq k<n} f_k^{-1}(B)$. If these preimages are pairwise disjoint we say that $B_{(n)}$ is an $n$\emph{-tower} for $f$ with \emph{base} $B$.

Our main result is the following multiparameter version of Theorem \ref{thm:Rokh} for non-invertible maps.

\begin{theorem}\label{thm:Rokhlin}
Let $\epsilon>0$ and let $n\in \N^d$. Then for every free measure-preserving action $f$ of $\N_0^d$ on an atomless standard probability space, there exists an $n$-tower for $f$ of measure at least $1-\epsilon$.
\end{theorem}

\noindent The proof given in \cite{H&S} for the case $d=1$ of this theorem uses  mostly elementary arguments, and involves also Zorn's lemma and the Poincar\'e recurrence theorem. We did not find a simple modification of this proof (or of the arguments in \cite{Conze,K&W}) yielding Theorem \ref{thm:Rokhlin}. We were also interested in whether Zorn's lemma could be avoided (note that this lemma is used also in several proofs of Theorem \ref{thm:Rokh} itself, for instance in \cite{CFS,K&M}). Our proof of Theorem \ref{thm:Rokhlin}, presented in Section \ref{sec:Rokhlin}, is completely elementary. 

In Section \ref{sec:apps} we discuss some applications of Theorem \ref{thm:Rokhlin}. In the setting of invertible maps, the applications of towers in ergodic theory are numerous and well documented (see \cite{Korn,Weiss}). Some of these results involving $\Z^d$-actions may be extended to non-invertible maps using Theorem \ref{thm:Rokhlin}, but for this paper we have chosen to treat different applications, of more recent origin and of combinatorial nature. The simplest one concerns the problem of finding solutions to an equation $c_1x_1=c_2x_2$ with integer coefficients $c_i$ and with variables $x_i$ lying in a given subset of a compact abelian group. A central quantity related to this problem is the following.
\begin{defn}\label{def:2vareqdens}
Let $c_1,c_2$ be non-zero integers and let $G$ be a compact abelian group with Haar probability $\mu$ on the Borel $\sigma$-algebra $\cB_G$. We say a set $A\in \cB_G$ is \emph{$(c_1,c_2)$-free} if there are no solutions $(x_1,x_2)\in A^2$ to the equation $c_1 x_1= c_2 x_2$. We define
\begin{equation}\label{eq:2vareqdens}
d_{(c_1,c_2)}(G)=\sup \big\{\mu(A): A\subset G\textrm{ is }(c_1,c_2)\textrm{-free}\big\}.
\end{equation}
\end{defn}
\noindent Using \cite[Theorem 2.5]{H&S} (the case $d=1$ of Theorem \ref{thm:Rokhlin}), Fiz-Pontiveros showed that for the circle group $\T=\R/\Z$ one has $d_{(1,\lambda)}(\T)=1/2$ for every non-zero integer $\lambda \neq 1$ (see \cite[Proposition 3.2]{Fiz}). The case $d=2$ of Theorem \ref{thm:Rokhlin} enables us to extend this result as follows (in particular this answers   \cite[Question 2]{Fiz}).
\begin{proposition}\label{prop:2vareqsT}
Let $c_1,c_2$ be distinct non-zero integers. Then $d_{(c_1,c_2)}(\T)=1/2$. 
\end{proposition}
\noindent The result holds for more general groups; see Proposition \ref{prop:2vareqs}. This result is in fact a simple application of a  general connection that Theorem \ref{thm:Rokhlin} establishes between a certain natural combinatorial problem concerning free measure-preserving actions of $\N_0^d$ and a problem concerning subsets $A$ of $\Z^d$ whose difference set $A-A$ avoids a prescribed finite set. This connection is developed in Subsection 3.2. We then relate this further to a similar problem on the circle group. Through this connection we obtain, in particular, the following generalization of Proposition \ref{prop:2vareqsT}.

\begin{proposition}\label{prop:bip2vareqs} 
Let $c_0=1$, let $c_1,\ldots,c_d$ be multiplicatively independent non-zero integers,\footnote{This means that if $c_1^{k_1}\cdots c_d^{k_d}=1$ with $k_i\in \Z$, then $k_i=0$ for every $i$.} and let $\Gamma$ be a bipartite graph on $\{0,1,\dots,d\}$. Then for every $\epsilon>0$ there is a Borel set $A\subset \T$ such that $\mu(A)\geq 1/2-\epsilon$ and $A$ is $(c_i,c_j)$-free for every edge $ij$ in $\Gamma$.
\end{proposition}

\noindent The value $1/2$ is clearly optimal for non-empty bipartite graphs. Proposition \ref{prop:bip2vareqs} is a special case of a similarly optimal result that we obtain concerning the more general class of \emph{star-extremal} graphs; see Proposition \ref{cor:bip2vareqs}.

\section{Proof of Theorem \ref{thm:Rokhlin}}\label{sec:Rokhlin}

Given a measure-preserving action $f$ of $\N_0^d$ on $X$, and $n\in \N^d$, we say that a set $B\in \cX$ is $n$-\emph{admissible} if it is a base of an $n$-tower for $f$.\\
\indent Our starting point is the following result, which we shall then iterate in order to find $n$-admissible sets of positive measure.

\begin{proposition}\label{prop:2-tower}
Let $f,g$ be commuting endomorphisms on $(X,\cX,\mu)$ satisfying
\[
\mu(\{x\in X: f(x)=g(x)\})=0.
\]
Then for every set $Y\in\cX$ and every $\epsilon>0$, there exists a measurable set $B \subset f^{-1}(Y)$ satisfying $\mu(B)\geq \frac{1}{4}(\mu(Y)-\epsilon)$ and $f^{-1}(B) \cap g^{-1}(B)=\emptyset$. 
\end{proposition}
Recall that the atomless probability space $(X,\cX,\mu)$ is isomorphic, modulo a null set, to the interval $[0,1]$ with Lebesgue measure \cite[Theorem 9.4.7]{Bog2}. In particular, for some set $X'\in \cX$ with $\mu(X')=1$,  there is a sequence $\cP_0,\cP_1, \cP_2, \ldots$ of finite measurable partitions of $X'$ with the following properties: we have $\cP_0=\{X'\}$; for each $r$ the partition $\cP_{r+1}$ refines $\cP_r$ (that is every atom of $\cP_r$ is a union of atoms of $\cP_{r+1}$); the sequence $(\cP_r)$ separates the points of $X'$, that is, for every $x\neq y$ in $X'$ there exists $r$ and distinct atoms $A,B$ in $\cP_r$ such that $x\in A, y\in B$.

We shall use the following fact.
\begin{lemma}\label{lem:fp}
Let $F=\{x\in f^{-1}(X'): f(x)=g(x)\}$, and for each positive integer $r$ let $\Omega_r=\bigsqcup_{A\in \cP_r} f^{-1}(A)\cap g^{-1}(A)$. 
Then $F = \bigcap_{r\in \N} \Omega_r$. In particular, we have $\mu(\Omega_r)\to 0$ as $r\to \infty$.
\end{lemma}
\begin{proof}
We have $F\subset \Omega_r$ for every $r$, so $F\subset \bigcap_{r\in \N} \Omega_r$. Since the sequence $(\cP_r)$ separates the points of $X'$, we also have $F\supset \bigcap_r \Omega_r$. Finally, note that $\Omega_r\supset \Omega_{r+1}$ for each $r$, since $\cP_{r+1}$ refines $\cP_r$. Hence $\mu(\Omega_r)\to \mu(F) =0$.
\end{proof}

\begin{proof}[Proof of Proposition \ref{prop:2-tower}]
Let $Y'=X'\cap Y$ and for each positive integer $r$ let $\cQ_r$ denote the partition of $Y'$ induced by $\cP_r$, namely the partition into sets $Q=P\cap Y'$, $P\in \cP_r$.\\
\indent It suffices to show that there is a measurable set $D\subset Y'$ satisfying 
\begin{equation}\label{eq:smallshift}
\mu\big(f^{-1} (D)\setminus g^{-1} (D)\big)\geq \frac{1}{4}(\mu(Y)-\epsilon).
\end{equation}
Indeed, if this holds then the measurable set $B=f^{-1}(D)\setminus g^{-1} (D)$ has the required properties, in particular we have $f^{-1}(B)\subset X\setminus g^{-1}f^{-1}(D)$ whereas $g^{-1}(B)\subset g^{-1}f^{-1}(D)$, whence $f^{-1}(B)\cap g^{-1}(B) =\emptyset$.\\
\indent To see that such a set $D$ exists, fix an arbitrary positive integer $r$, and let $D_r\subset Y'$ be generated randomly by letting each set $A\in \cQ_r$ be contained in $D_r$ independently with probability $1/2$ (and contained in $Y'\setminus D_r$ otherwise). We have
\begin{eqnarray*}
\E_{D_r}\; \mu\big(f^{-1}(D_r)\setminus g^{-1} (D_r)\big) & = & \E_{D_r}\; \mu\big(f^{-1}(D_r)\big)-\E_{D_r}\; \mu\big(f^{-1}(D_r)\cap g^{-1} (D_r)\big)\\
& = & \frac{\mu\big(Y\big)}{2}-\sum_{A,B\in \cQ_r} \P\big((A\cup B)\subset D_r\big)\; \mu\big(f^{-1}(A)\cap g^{-1} (B)\big).
\end{eqnarray*}
The last sum equals
\begin{eqnarray*}
&&\frac{1}{4}\sum_{A\neq B\in \cQ_r} \mu\big(f^{-1}(A)\cap g^{-1} (B)\big)+\frac{1}{2}\sum_{A\in \cQ_r} \mu\big(f^{-1}(A)\cap g^{-1} (A)\big)\\
& = &\frac{1}{4}\sum_{A,B\in \cQ_r} \mu\big(f^{-1}(A)\cap g^{-1} (B)\big)+\frac{1}{4}\sum_{A\in \cQ_r} \mu\big(f^{-1}(A)\cap g^{-1} (A)\big)\\
& = &\frac{1}{4}\mu\big(f^{-1}(Y)\cap g^{-1} (Y)\big)+\frac{1}{4}\sum_{A\in \cQ_r} \mu\big(f^{-1}(A)\cap g^{-1} (A)\big)\\
& \leq &\frac{1}{4}\big(\mu\big(Y\big)+\mu (\Omega_r)\big).
\end{eqnarray*}
Therefore
$\E_{D_r}\; \mu\big(f^{-1}(D_r)\setminus g^{-1} (D_r)\big) \geq  \frac{1}{4}\big(\mu\big(Y\big)-\mu (\Omega_r)\big)$. It follows that for each $r\in \N$ there exists a measurable set $D_r\subset Y'$ such that
\[
\mu\big(f^{-1}(D_r)\setminus g^{-1} (D_r)\big) \geq  \frac{1}{4}\big(\mu\big(Y\big)-\mu (\Omega_r)\big).
\]
By Lemma \ref{lem:fp}, we can satisfy \eqref{eq:smallshift} with $D=D_r$ for $r=r(\epsilon,f,g)$ sufficiently large.
\end{proof}
Given $Y\in \cX$, we now iterate Proposition \ref{prop:2-tower} to find, in some preimage of $Y$, an $n$-admissible set of measure proportional to $\mu(Y)$.

For $n=\big(n(1),\dots,n(d)\big)\in \N^d$ we denote by $\pi(n)$ the product $n(1)\cdots n(d)$.

\begin{lemma}\label{lem:basic-n-tower}
Let $n\in \N^d$ and let $Y\in \cX$. Then, for some $N\in \N_0^d$, there exists an $n$-admissible set $B \subset f_N^{-1}(Y)$ such that $\mu(B) \geq \mu(Y)/5^{\binom{\pi(n)}{2}} $.
\end{lemma}

\begin{proof}
Let us fix any distinct $k_1, \ell_1\in \N_0^d$ with $k_1, \ell_1 <n$. We  apply Proposition \ref{prop:2-tower} with $f=f_{k_1},g=f_{\ell_1}$, to obtain $B_1 \subset f_{k_1}^{-1}(Y)$ with $\mu(B_1) \geq \mu(Y)/5$ and $f_{k_1}^{-1}(B_1) \cap f_{\ell_1}^{-1}(B_1)=\emptyset$. Now we apply the proposition again with $Y=B_1$ and $f_{k_2},f_{\ell_2}$ for some other pair of distinct $k_2, \ell_2<n$, to obtain $B_2 \subset f_{k_2}^{-1}(B_1)\subset f_{k_1+k_2}^{-1}(Y)$ such that $\mu(B_2) \geq \mu(B_1)/5$ and $f_{k_2}^{-1}(B_2)\cap f_{\ell_2}^{-1}(B_2)=\emptyset$.  Proceeding in this way for each of the remaining pairs $k,\ell <n$, the result follows.
\end{proof}
We shall now enhance Lemma \ref{lem:basic-n-tower}, by showing that the measure of the $n$-tower $B_{(n)}$ can be guaranteed to be at least a fixed fraction (independent of $n$) of the measure of the original set $Y$.
\begin{lemma}\label{lem:n-adm}
Let $n \in \N^d$ and let $Y\in \cX$. Then for every $\epsilon>0$ there exists $N\in\N_0^d$ and an $n$-admissible set $B\subset f_N^{-1}(Y)$ such that $\mu\big(B_{(n)}\big)> 2^{-d}\mu(Y)-\epsilon$.
\end{lemma}

\begin{proof}
Let $\rho=\sup\{\mu(B): B\subset f_N^{-1}(Y)\textrm{ is }n\textrm{-admissible},\, N\in \N_0^d\}$ and suppose for a contradiction that $\rho<\mu(Y)/(2^d \pi(n))$. Let $\delta=\mu(Y)-2^d \,\pi(n)\,\rho>0$. By definition of $\rho$, there exists $N_0\in \N_0^d$ and an $n$-admissible set $D\subset f_{N_0}^{-1}(Y)$ satisfying  $\mu(D)>\rho-\delta\, 5^{-\binom{\pi(n)}{2}}$. Let $Y'=f_{N_0+n}^{-1}(Y)$ and $B'=f_n^{-1}(D)\subset Y'$.\\
\indent By Lemma \ref{lem:basic-n-tower} applied to  $Y' \setminus D_{(2 n)}$,  where $2n=\big(2n(1),2n(2),\ldots,2n(d)\big)$, there exists $N_1\in \N_0^d$ and an $n$-admissible set $B'' \subset f_{N_1}^{-1}(Y' \setminus D_{(2n)}) \subset f_{N_1}^{-1}(Y' \setminus B'_{(n)})$ such that $\mu(B'') \geq \mu(Y'\setminus D_{(2n)})\,5^{-\binom{\pi(n)}{2}}\geq \delta\,5^{-\binom{\pi(n)}{2}}$. Let $B=f_{N_1}^{-1}(B' )\sqcup B''$, and note that $B\subset f_N^{-1}(Y)$ where $N=N_0+N_1+n$.\\
\indent We claim that $B$ is $n$-admissible. Indeed, for any distinct $k,\ell\in \N_0^d$ with $k,\ell<n$, we have
\begin{eqnarray*}
f_k^{-1}(B)\cap f_\ell^{-1}(B) & = & \big(f_{N_1+k}^{-1}(B' )\sqcup f_k^{-1}(B'')\big)\; \cap \; \big(f_{N_1+\ell}^{-1}(B' )\sqcup f_\ell^{-1}(B'')\big)\\
& = & \big(f_{N_1+k}^{-1}(B' )\cap f_{N_1+\ell}^{-1}(B' )\big)\; \sqcup \; \big(f_k^{-1}(B'')\cap f_{N_1+\ell}^{-1}(B')\big)\\
&& \hspace{0.2cm}\;\sqcup\; \big(f_{N_1+k}^{-1}(B' )\cap f_\ell^{-1}(B'')\big)\;\sqcup\; \big(f_k^{-1}(B'')\cap f_\ell^{-1}(B'')\big).
\end{eqnarray*}
Here the first and fourth intersections are empty, since $f_{N_1}^{-1}(B')$ and $B''$ are $n$-admissible. The second intersection is also  empty, since $f_k^{-1}(B'')$ lies in the complement of $f_{N_1+k}^{-1}(D_{(2n)})$ while $f_{N_1+\ell}^{-1}(B')$ lies in $f_{N_1+n+\ell}^{-1}(D)\subset f_{N_1+k}^{-1}(D_{(2n)})$. Similarly, the third intersection is empty, so our claim holds.\\
\indent Thus we have obtained a set $B\subset f_N^{-1}(Y)$ that is $n$-admissible and that satisfies $\mu(B)= \mu(D)+\mu(B'')>\rho$, a contradiction.
\end{proof}
We can now prove our main result.

\begin{proof}[Proof of Theorem \ref{thm:Rokhlin}]
For each $N \in \N^d$, let
\[
c_N=\sup\{\mu(B_{(N)}): B\in \cX \textrm{ is }N\textrm{-admissible}\}.
\]
For each $n\in \N^d$, the sequence $(c_{2^kn})_{k\in \N}$ is decreasing. Indeed, given any $k\in \N$, if $A$ is a $2^{k+1} n$-admissible set, then the following set is $2^k n$-admissible:
\[
B= \bigsqcup_{t_1,t_2,\ldots,t_d\, \in\,\{0,1\}} f_{\big(t_1 2^k n(1),\ldots,t_d 2^k n(d)\big)}^{-1}(A),
\]
and so $\mu(A_{(2^{k+1} n)})=\mu(B_{(2^k n)})\leq c_{2^k n}$, whence $c_{2^{k+1} n}\leq c_{2^k n}$.\\
\indent Now fix any $n\in \N^d$ and let $c=\inf_{k\in \N} c_{2^kn}$.  By Lemma \ref{lem:n-adm} applied with $Y=X$, we have $c \geq 2^{-d}$. We shall prove that $c\geq 1$.\\
\indent Suppose for a contradiction that $c<1$, fix $N=2^k n$ for an arbitrary $k\in \N$, and fix an arbitrary $\delta\in (0,c)$. Let $K \in \N$ be sufficiently large so that firstly $d\,2^{-K}\leq \delta$ and secondly there exists a $2^K N$-admissible set $B'$ satisfying
\begin{equation}\label{eq:bound}
\big |\,\mu\big(B'_{(2^K N)}\big)-c\,\big|\leq \delta.
\end{equation}
Let $Y'=X \setminus B_{(2^K N)}'$. 
By Lemma \ref{lem:n-adm} applied with $Y=Y'$, there exists $N'$ and an $N$-admissible set $B''\subset f_{N'}^{-1}(Y')$ satisfying 
$\mu(B_{(N)}'')\geq 2^{-d-1}\mu(Y')$.\\
\indent Let $D$ be the following $N$-admissible set:
\[
D=\bigsqcup_{t_1,\ldots,t_d\, \in\, \{1,\dots,2^K-1\}} f_{\big(t_1 N(1),\ldots,t_d N(d)\big)+N'}^{-1}(B').
\]
Note that $D$ lies in $f_{N'}^{-1}(B'_{(2^KN)})$ and is therefore disjoint from $B''$. Let
\[
B= B''\;\sqcup\;D.
\]
We claim that $B$ is $N$-admissible. Indeed, for every distinct $i,j<N$, we have
\begin{eqnarray*}
f_i^{-1}(B)\cap f_j^{-1}(B) & = & (f_i^{-1}(B'')\sqcup f_i^{-1}(D))\cap (f_j^{-1}(B'')\sqcup f_j^{-1}(D))\\
& = & [f_i^{-1}(B'')\cap f_j^{-1}(B'')]\sqcup  [f_i^{-1}(D)\cap f_j^{-1}(B'')]\\
&& \hspace{0.3cm}\;\sqcup\; [f_i^{-1}(B'')\cap f_j^{-1}(D)]\sqcup  [f_i^{-1}(D)\cap f_j^{-1}(D)].
\end{eqnarray*}
Here the first and fourth intersections are empty since $B'', D$ are both $N$-admissible. The second intersection is empty, for if there existed $z\in f_i^{-1}(D)\cap f_j^{-1}(B'')$ then $f_j(z)$ would lie in $B''\subset X\setminus f_{N'}^{-1}(B'_{(2^KN)})$, yet we would also have 
\begin{eqnarray*}
f_j(z)\in f_j(f_i^{-1}(D)) & \subset & \bigcup_{t_1,\ldots,t_d\, \in\, \{1,\dots,2^K-1\}} f_j(f_{(t_1 N(1),\ldots,t_d N(d))+N'+i}^{-1}(B'))\\
& \subset &  \bigcup_{t_1,\ldots,t_d\, \in\, \{1,\dots,2^K-1\}} f_{(t_1 N(1),\ldots,t_d N(d))+N'+i-j}^{-1}(B')\\
& \subset & f_{N'}^{-1}(B_{(2^KN)}'),
\end{eqnarray*}
a contradiction. Similarly the third intersection is empty, so $B$ is indeed $N$-admissible.\\
The fact that $f_i^{-1}(D), f_j^{-1}(B'')$ are disjoint for all $0\leq i,j<N$ also implies that $B_{(N)}''$ and $D_{(N)}$ are disjoint.\\
\indent We have thus obtained an $N$-admissible set $B$ satisfying
\[
\mu(B_{(N)})= \mu(B_{(N)}'') + \mu(D_{(N)}) \geq 2^{-d-1}\mu(Y')+ \mu(B'_{(2^K N)})-\mu(B'_{(N)})\,d\,(2^K)^{(d-1)}.
\]
Since $\mu(B'_{(N)})=\mu(B'_{(2^K N)})\,2^{-Kd}$, we therefore have, using \eqref{eq:bound}, that
\[
c_N  \geq  2^{-d-1}\mu(Y')+ \mu(B'_{(2^K N)})(1-d\,2^{-K})
 \geq  2^{-d-1}(1-c-\delta) + (c-\delta)(1-\delta).
\]
Since $\delta$ was arbitrary, we deduce that $c_{2^k n} \geq c+2^{-d-1}(1-c)$, and since $k$ was arbitrary, we deduce that $c \geq c+2^{-d-1}(1-c)>c$, a contradiction.
\end{proof}
\section{Applications}\label{sec:apps}
\subsection{2-variable equations on compact abelian groups}
\noindent A central topic in additive combinatorics consists in determining the greatest size that a subset of an abelian group can have without containing solutions to a given integer linear equation. The simplest non-trivial case is that of a 2-variable homogeneous equation, which we write in general form as $c_1 x_1 = c_2 x_2$, for fixed non-zero integer coefficients $c_1,c_2$.  For a compact abelian group $G$ and an integer $n$, let $T_n\colon G\to G$, $x\mapsto n x$. For  $A\subset G$ let $T_n(A)=\{n\,a:a\in A\}$. We consider the problem of determining the quantity $d_{(c_1,c_2)}(G)$ from Definition \ref{def:2vareqdens}
, that is
\[
d_{(c_1,c_2)}(G)=\sup \big\{\mu(A): A\in \cB_G,\; T_{c_1}(A)\cap T_{c_2}(A)=\emptyset\big\}.
\]
\noindent If $T_{c_1},T_{c_2}$ are both surjective then they preserve $\mu$ and,  since $T_{c_i}^{-1} T_{c_i} A\supset A$, we then have $\mu(T_{c_i}(A))\geq \mu(A)$; it follows that $d_{(c_1,c_2)}(G)\leq 1/2$. Theorem \ref{thm:Rokhlin} yields a simple proof that in fact $d_{(c_1,c_2)}(G)= 1/2$ under quite general conditions on $G$. More precisely, this holds provided that the triple $(G,\cB_G,\mu)$ yields an atomless standard probability space and that the endomorphisms $f_n=T_{c_1}^{n(1)}\circ T_{c_2}^{n(2)}$ form a free action of $\N_0^2$ on this space. We record this as follows.

\begin{proposition}\label{prop:2vareqs}
Let $G$ be a polish divisible compact abelian group. Then for every  distinct non-zero integers $c_1,c_2$, we have $d_{(c_1,c_2)}(G)=1/2$. 
\end{proposition}
\begin{proof}
We can suppose that $c_1,c_2$ are coprime; indeed, given a common divisor $\ell$, with $c_i=\ell c_i'$ for $i=1,2$, we have by the divisibility of $G$ that $d_{(c_1,c_2)}(G)= d_{(c_1',c_2')}(G)$.\\
\indent Let us suppose first that both $|c_1|,|c_2|$ are greater than 1. Then by unique factorization, for every distinct elements $m,n\in  \N_0^2$, we have $c_1^{m(1)}c_2^{m(2)}\neq c_1^{n(1)}c_2^{n(2)}$, and it follows that
\[
\mu\big(\{x\in G: c_1^{m(1)}c_2^{m(2)}x = c_1^{n(1)}c_2^{n(2)} x\}\big)= \mu\big(T_{c_1^{m(1)}c_2^{m(2)}-c_1^{n(1)}c_2^{n(2)}}^{-1}\{0\}\big)=0.
\]
We can therefore apply Theorem \ref{thm:Rokhlin} to the free action $f$ of $\N_0^2$ on $G$, where $f_n(x)=T_{c_1}^{n(1)}\circ T_{c_2}^{n(2)}\,x$. Fix $\delta>0$, and apply the theorem with $\epsilon=\delta/2$ and $N=(t,2)$ with $t>1/\delta$. Let $A$ be the $N$-admissible set given by the theorem, with $\mu(A_{(N)})\geq 1-\delta/2$. Now let
\begin{equation}\label{eq:Bset}
B=\bigsqcup_{j\in [t-1]} (T_{c_1}^j\circ T_{c_2})^{-1}(A).
\end{equation}
(Recall that $[t-1]=\{1,\ldots,t-1\}$.) We have $\mu(B)\geq (1-\delta/2)/2-1/(2t)\geq 1/2-\delta $. We also have $T_{c_1}(B)\cap T_{c_2}(B)=\emptyset$. Hence $d_{(c_1,c_2)}(G)\geq 1/2-\delta$. Since $\delta$ was arbitrary, the result follows.\\
\indent When one of $|c_1|,|c_2|$ equals 1, the case $d=1$ of Theorem \ref{thm:Rokhlin} implies immediately that $d_{(c_1,c_2)}(G)=1/2$.
\end{proof}
\begin{remark}
The supremum $1/2$ in Proposition \ref{prop:2vareqs} need not be attainable. This was already observed in \cite{Fiz} in the case $c_1=1$, $|c_2|>1$: for instance, for $G=\T$, attainment of this supremum would contradict the ergodicity of $T_{c_2^2}$.
\end{remark}
\noindent In the proof above, Theorem \ref{thm:Rokhlin} is used to reduce the problem to that of finding a set $S$ of maximal density inside a rectangle in $\Z^2$ such that some translates of $S$ are disjoint, namely the translates $S-(1,0)$ and $S-(0,1)$. In the next subsection we elaborate on this use of Theorem \ref{thm:Rokhlin} to obtain a more general connection between two very natural problems.
\subsection{Sets with some disjoint images, and a problem of Motzkin}
\noindent The first problem in question here concerns general free measure-preserving actions of $\N_0^d$. 

\begin{defn}
Let $f$ be a free action of $\N_0^d$ on $(X,\mathcal{X},\mu)$, let $V$ be a finite subset of $\N_0^d$, and let $\Gamma$ be a graph with vertex set $V$. We say that a measurable set $A\in\cX$ is $\Gamma$-\emph{admissible for} $f$ if for every edge $u\,v$ in $\Gamma$ we have $f_u(A)\cap f_v(A)=\emptyset$. We define
\[
d_\Gamma(X,f) = \sup \big\{\mu(A): A\textrm{ is $\Gamma$-admissible for }f\big\}.
\]
\end{defn}
\noindent The general problem consists in determining $d_\Gamma(X,f)$. This includes the following problem, which extends the one treated in the previous subsection.
\begin{example}[Avoiding several 2-variable equations]\label{ex:2vareqsfam}
Let $m_1,m_2,\ldots,m_d$ be multiplicatively independent non-zero integers. Given a polish divisible compact abelian group $G$, for each $n\in \N_0^d$ let $f_n:G\to G$, $x\mapsto m_1^{n(1)}\cdots m_d^{n(d)} x$. One checks from the definitions that these maps form a free measure-preserving action $f$ of $\N_0^d$ on $(G,\cB_G,\mu)$. Now let $\mathcal{F}$ be a finite family of 2-variable  equations $c_i x_1=c_j x_2$ with non-zero integer coefficients $c_i$, in which every coefficient is of the (unique) form $m_1^{v(1)}\cdots m_d^{v(d)}$ for some $v\in \N_0^d$. Let $V$ be the subset of $\N_0^d$ formed by these elements $v$, and let $\Gamma$ be the graph on $V$ defined by letting $u\,v$ be an edge if and only if the equation $m_1^{u(1)}\cdots m_d^{u(d)}\;x_1=m_1^{v(1)}\cdots m_d^{v(d)}\;x_2$ is in $\mathcal{F}$. Call a set $A\subset G$ an $\mathcal{F}$\emph{-free set} if there are no solutions in $A^2$ to any of the equations in $\mathcal{F}$. The problem is to determine $d_{\mathcal{F}}(G):=\sup\{\mu(A):A\in \cB_G\textrm{ is $\mathcal{F}$-free}\}$. Note that $d_{\mathcal{F}}(G)=d_\Gamma(G,f)$.
\end{example}

\noindent Recall that for $N\in \N^d$ we denote the product $N(1)\cdots N(d)$ by $\pi(N)$. 

\begin{defn}
Let $\Gamma$ be a graph on some finite subset $V$ of $\N_0^d$. We say that a set $S\subset \Z^d$ is $\Gamma$\emph{-admissible} if for every edge $u\,v$ of $\Gamma$ the translates $S-u$, $S-v$ are disjoint. We define
\[
M_\Gamma(N)=\max \Big\{|S|/\pi(N): S\subset \prod_{i=1}^d\big[0,N(i)\big)\textrm{ is $\Gamma$-admissible} \Big\}.
\]
\end{defn}
\noindent Theorem \ref{thm:Rokhlin} enables us to relate $d_\Gamma(X,f)$ and $M_\Gamma(N)$ as follows. 
\begin{proposition}\label{prop:latticereduction}
Let $\Gamma$ be a graph on some finite subset $V$ of $\N_0^d$, let $m\in \N^d$ be such that $m>v$ for all $v\in V$, and let $N\in \N^d$. Then for every free action $f$ of $\N_0^d$ on $(X,\cX,\mu)$, we have
\begin{equation}\label{eq:latticebound}
\frac{\pi(N)}{\pi(N+m)}\;M_\Gamma(N)\;  \leq\; d_\Gamma(X,f)\; \leq \; M_\Gamma(N) + \sum_{i\in [d]} \frac{1}{N(i)}.
\end{equation}
\end{proposition}
 \noindent As a consequence we have the following result, to the effect that $d_\Gamma(X,f)$ does not depend on $X,f$.
\begin{corollary}\label{cor:effect}
Let $\Gamma$ be a graph on some finite subset $V$ of $\N_0^d$. Then for every free action $f$ of $\N_0^d$ on a space $(X,\cX,\mu)$, we have that $M_\Gamma (N_j)$ converges to $d_\Gamma(X,f)$ for every sequence $(N_j)$ of elements of $\N^d$ satisfying $\min_{i\in [d]}N_j(i)\to\infty$ as $j\to\infty$.
\end{corollary}

\begin{proof}[Proof of Proposition \ref{prop:latticereduction}]
We begin with the inequality on the left in \eqref{eq:latticebound}. Given an arbitrary $\epsilon>0$, we apply Theorem \ref{thm:Rokhlin} to obtain an $(N+m)$-tower for $f$ with base $B$ and measure at least $1-\epsilon$. Let $S\subset \prod_{i\in [d]} \big[0,N(i)\big)$ be $\Gamma$-admissible with $|S|= M_\Gamma(N)\,\pi(N)$, and let $A=\bigsqcup_{n\in S+m} f_n^{-1}(B)$. Since the shifted sets $S+m-v$, $v\in V$, are all contained in $\prod_{i\in [d]} \big[0,N(i)+m(i)\big)$, and $(S+m-u)\cap (S+m-v)=\emptyset$ for every edge $uv$ in $\Gamma$, the set $A$ must be $\Gamma$-admissible for $f$. Hence
\[
d_\Gamma(X,f)\geq \mu(A)\geq (1-\epsilon)\frac{ |S|}{\pi(N+m)}\geq M_\Gamma(N)\,\frac{\pi(N)}{\pi(N+m)}-\epsilon.
\]
Letting $\epsilon \to 0$, the inequality follows.\\
\indent For the inequality on the right in \eqref{eq:latticebound}, fix again an arbitrary $\epsilon>0$, and let $A\in \cX$ be $\Gamma$-admissible for $f$ with $\mu(A)\geq d_\Gamma(X,f)-\epsilon/2$. Theorem \ref{thm:Rokhlin} gives us a base $B$ of an $N$-tower for $f$ of measure at least $1-\epsilon/2$. We define the following measurable function on $X$:
\[
F(x) = \sum_{0< n< N} 1_A(f_n(x)).
\]
The integral of $F$ over $f_N^{-1}(B)$ equals roughly $\mu(A)$. Indeed, we have
\begin{eqnarray*}
\int_{f_N^{-1}(B)} F(x)\ud\mu(x) &  =  & \sum_{0<n< N} \int_X 1_A(f_n(x))\, 1_B(f_{N-n}\circ f_n  (x)) \ud\mu(x)\\
&  =  & \sum_{0< n< N} \mu\big(A \cap f_{N-n}^{-1} (B)\big)\\
& \geq & \mu(A\cap B_{(N)}) - \pi(N)^{-1}\,|\{0\leq n<N:\exists\, i\in [d],\; n(i)=0\}|\\
& \geq  & \mu(A)-\frac{\epsilon}{2}-\sum_{i\in [d]} \frac{1}{N(i)}.
\end{eqnarray*}
Since $\mu(f_N^{-1}(B))\leq \pi(N)^{-1}$, we conclude that there exists $x^*\in f_N^{-1}(B)$ such that
\[
\pi(N)^{-1} F(x^*)\geq \mu(A)-\frac{\epsilon}{2}-\sum_{i\in [d]} \frac{1}{N(i)}.
\]
We now set
\[
S:=\{n : 0<n<N,\; f_n(x^*) \in A\}.
\]
Suppose that for some edge $u\,v$ in $\Gamma$ we had $(S-u)\cap (S-v)\neq\emptyset$, so that $n_1-u=n_2-v$ for some $n_1,n_2\in S$. Let $x_1=f_{n_1}(x^*) \in A$ and $x_2=f_{n_2}(x^*) \in A$. Then we have $f_u(x_2)=f_{n_2+u}(x^*)= f_{n_1+v}(x^*)=f_v(x_1)$, so $f_u(A)\cap f_v(A)\neq\emptyset$, a  contradiction. Therefore $S$ is $\Gamma$-admissible. We then have
\[
M_\Gamma(N)\geq |S|/\pi(N)=F(x^*)/\pi(N)\geq d_\Gamma(X,f)-\epsilon-\sum_{i\in [d]} \frac{1}{N(i)}.
\]
Letting $\epsilon \to 0$, the result follows.
\end{proof}
\noindent From now on we shall write $d_\Gamma$ for this quantity $d_\Gamma (X,f)=\lim_{j\to\infty}  M_\Gamma(N_j)$.

Given a set $V$ and a graph $\Gamma$ on $V$ as above, consider the (symmetric)  partial difference set of $V$ along $\Gamma$, that is the set\footnote{ This is a version for graphs of the partial difference set $V\stackrel{\Gamma}- V$ defined for bipartite graphs in \cite[\S 2.5]{T-V}.}
\[
D=\{u-v,\; v-u\,:\,u,v\in V,\; u\,v\textrm{ an edge of }\Gamma\}\subset \Z^d.
\]
A set $A\subset \Z^d$ is $\Gamma$-admissible if and only if the difference set $A-A$ is disjoint from $D$. Writing $\overline{\delta}(A)$ for the \emph{upper density} of $A$, that is $\overline{\delta}(A)= \limsup_{r\to \infty} |A\cap [-r,r)^d|/(2r)^d$, a straightforward argument shows that $d_\Gamma= \sup\{\overline{\delta}(A):A\subset \Z^d,\; (A-A) \cap D=\emptyset\}$.\\
\indent The general problem of determining the supremum of upper densities of sets $A\subset \Z^d$ with differences avoiding a given finite set goes back to Motzkin (who posed it originally for sets $A\subset \N$; see \cite{Cantor&G}). This problem is vast and we shall not explore it fully here. However, we shall give an estimate for $d_\Gamma$ for a family of graphs which, thanks to the connection with the quantities $d_\mathcal{F}(G)$ established in Example \ref{ex:2vareqsfam}, will yield in particular a nontrivial generalization of Proposition \ref{prop:2vareqs}, namely Proposition \ref{cor:bip2vareqs} below.\\
\indent The family just mentioned involves the graphs $\Gamma$ that have as vertex set $\{0,e_1,\ldots,e_d\}$, where the elements $e_i$ form the standard basis of $\R^d$.
\subsubsection{Estimation of $d_\Gamma$ for graphs $\Gamma$ on $\{0,e_1,\ldots,e_d\}$.}
\noindent Our aim here is to give bounds for $d_\Gamma$ in terms of known graph parameters. To this end, we first express $d_\Gamma$ as a natural quantity on the circle group $\T$.
\begin{defn}
Let $\Gamma$ be a finite graph, with vertex set $V$. We say that a Borel set $A\subset \T$ is a \emph{coloring base} for $\Gamma$ if there exists a map $\varphi:V\to \T$ such that for every edge $uv$ of $\Gamma$ we have $(A+\varphi(u)) \cap (A+\varphi(v))=\emptyset$. We denote by $\sigma_\T(\Gamma)$ the supremum over all probability Haar measures of coloring bases for $\Gamma$.
\end{defn}

\begin{lemma}\label{lem:sig=d}
For every graph $\Gamma$ on $V=\{0,e_1,\ldots,e_d\}$ we have $d_\Gamma=\sigma_\T(\Gamma)$.
\end{lemma}
\begin{proof}
Let us view $\T$ as $[0,1)$ with addition mod 1.\\
\indent To see that $d_\Gamma\leq \sigma_\T(\Gamma)$, let $\alpha_1,\dots,\alpha_d$ be real numbers in $[0,1)$ such that $1,\alpha_1,\dots,\alpha_d$ are independent over $\Q$. This implies that for every distinct $m,n\in \N_0^d$ we have
\[
m(1)\alpha_1+\cdots+m(d)\alpha_d\neq n(1)\alpha_1+\cdots+n(d)\alpha_d\mod 1.
\]
Therefore the maps $f_n:\T\to\T$, $x\mapsto n(1)\alpha_1+\cdots+n(d)\alpha_d+x$ form a free $\N_0^d$-action $f$ on $X=(\T,\mu)$. Corollary \ref{cor:effect} implies that $d_\Gamma=d_\Gamma(X,f)$. By definition we then have $d_\Gamma(X,f)\leq \sigma_\T(\Gamma)$. Indeed, if $A\subset \T$ is $\Gamma$-admissible for $f$ then let $\varphi:V\to \T$ be defined by $\varphi(0)=0$, $\varphi(e_i)=\alpha_i$.\\
\indent To see that $\sigma_\T(\Gamma)\leq d_\Gamma$, let us write $v_0=0$, $v_i=e_i$ for $i\in [d]$, and fix any $\epsilon>0$. Let $A\subset \T$ be a Borel set with $\mu(A)\geq \sigma_\T(\Gamma)-\epsilon/4$ and let $\varphi:V\to \T$ be such that $(A+\varphi(v_i)) \cap (A+\varphi(v_j))=\emptyset$ for every edge $v_iv_j$ of $\Gamma$. We may assume that $\varphi(v_0)=0$. By standard properties of the Lebesgue measure, for some $\ell\in \N$ there exists a set $A'\subset \T$ that is a union of some intervals of the form $[(j-1)/\ell,j/\ell)$ with $j\in [\ell]$, such that $\mu(A\Delta A')\leq\epsilon/ (8 e(\Gamma))$ (where $e(\Gamma)$ is the number of edges of $\Gamma$). Letting $A_i=A+\varphi(v_i)$, $A_i'=A'+\varphi(v_i)$, we have $\mu(A_i'\cap A_j')\leq \mu(A_i\cap A_j)+\mu(A_i'\Delta A_i)+\mu(A_j'\Delta A_j) \leq \epsilon/(4 e(\Gamma))$ for every edge $v_iv_j$. Now let $\alpha_1,\dots,\alpha_d$ be real numbers in $[0,1)$ such that $\alpha_1,\dots,\alpha_d,1$ are independent over $\Q$ and such that $|\alpha_i-\varphi(e_i)|\leq \epsilon/(16\,\ell e(\Gamma))$ for every $i\in [d]$. To see that such $\alpha_i$ exist, note first that if we fix any $\beta_1,\ldots,\beta_d \in [0,1)$ such that $1,\beta_1,\dots,\beta_d$ are independent over $\Q$, then for any positive integer $n$ the elements $n\beta_1 \mod 1,\ldots,n\beta_d \mod 1 \in [0,1)$ and 1 are also independent.  Moreover, the orbit $\{(n \beta_1,\dots, n\beta_d):n\in \N\}$ is dense in $\T^d$, by Kronecker's theorem. Hence we can set $\alpha_i=n\beta_i$ for some $n$. Letting $\alpha_0=0$, we have  $\mu\big( (A'+\alpha_i) \Delta A_i'\big)\leq 2\ell\, |\alpha_i-\varphi(e_i)|\leq \epsilon/(8e(\Gamma))$ for each $v_i\in V$. We deduce that $\mu\big( (A'+\alpha_i) \cap (A'+\alpha_j)\big)\leq \epsilon/(2e(\Gamma))$ for every edge $v_iv_j$. Removing the unwanted intersections from $A'$, it follows that there is a set $A''\subset A'$ of measure at least $\mu(A')-\epsilon/2$ which is $\Gamma$-admissible for the free $\N_0^d$-action on $\T$ generated by the translations $x\mapsto x+\alpha_i$. Hence $d_\Gamma\geq \sigma_\T(\Gamma)-\epsilon$.
\end{proof}
\indent The \emph{circular chromatic number} (or star-chromatic number) of a finite graph $\Gamma$, denoted $\chi_c(\Gamma)$, is the infimum over all real numbers $q$ such that for each vertex $v$ of $\Gamma$ there exists an open interval  $A_v\subset \T$ with $\mu(A_v)=1/q$, such that $A_u\cap A_v=\emptyset$ for every edge $uv$ in $\Gamma\;$ (see \cite{Zhu}).\\
\indent The \emph{fractional chromatic number} of $\Gamma$, denoted $\chi_f(\Gamma)$, is the infimum over all real numbers $q$ such that for each vertex $v$ there exists a measurable set $A_v\subset [0,1)$ with $\mu(A_v)=1/q$ such that $A_u\cap A_v=\emptyset$ for every edge $uv$ in $\Gamma$.\\
\indent It is a standard fact that $\chi_c(\Gamma),\chi_f(\Gamma)$ are both rational numbers satisfying
\[
\omega(\Gamma)\leq \chi_f(\Gamma)\leq \chi_c(\Gamma) \leq \lceil \chi_c(\Gamma) \rceil = \chi(\Gamma),
\]
where $\omega(\Gamma),\chi(\Gamma)$ are the clique number and chromatic number respectively.\\
\indent From the definitions it follows that for every finite graph $\Gamma$ we have
\begin{equation}\label{eq:graphparams}
1/\chi_c(\Gamma)\leq \sigma_\T(\Gamma)\leq 1/\chi_f(\Gamma).
\end{equation}
\noindent If $\chi_f(\Gamma)=\chi_c(\Gamma)$ then $\Gamma$ is said to be \emph{star-extremal} (see \cite[\S 6]{Zhu}).\\

\indent We deduce the following result, of which Proposition \ref{prop:bip2vareqs} is a special case.
\begin{proposition}\label{cor:bip2vareqs}
Let $c_0=1$, let $c_1,\ldots,c_d$ be multiplicatively independent non-zero integers, let $\Gamma$ be a star-extremal graph on $\{0,1,\ldots,d\}$, and let $\mathcal{F}$ be the family of 2-variable equations $\{c_ix_1=c_jx_2: \; ij\textrm{ an edge of }\Gamma\}$. Then for every polish divisible compact abelian group $G$ we have $d_{\mathcal{F}}(G)=1/\chi_c(\Gamma)= 1/\chi_f(\Gamma)$.
\end{proposition}
\begin{proof}
Consider the free action $f$ of $\N_0^d$ on $G$ generated by the maps $x\mapsto c_i x$, thus for each $n\in \N_0^d$ we have $f_n:G\to G$, $x\mapsto c_1^{n(1)}\cdots c_d^{n(d)} x$. By the argument from Example \ref{ex:2vareqsfam}, applied here with $V=\{0,e_1,\dots,e_d\}$, we have $d_{\mathcal{F}}(G)=d_\Gamma(G,f)=d_\Gamma$. Combining this with Lemma \ref{lem:sig=d} and \eqref{eq:graphparams}, the result follows.
\end{proof}

\indent We shall not pursue further in this paper the problem of determining $d_\Gamma$ for more general graphs. Let us end with a question related to this problem.\\
\indent A set $S\subset \Z^d$ is said to be \emph{periodic} if there exist linearly independent vectors $v_1,\ldots,v_d \in \Z^d$ such that $S+v_i=S$ for all $i\in [d]$.

\begin{question}
Does there exist, for each graph $\Gamma$ on a finite subset of $\N_0^d$, a  periodic $\Gamma$-admissible set $S\subset \Z^d$ with density equal to $d_\Gamma$?
\end{question}

\noindent Equivalently, does there exist, for each finite set $D\subset \Z^d$, a periodic set $S\subset \Z^d$ with $(S-S)\cap D =\emptyset$ and of density equal to $\nu(D):=\sup\{\overline{\delta}(A):A\subset \Z^d, (A-A)\cap D=\emptyset\}$? (The case $d=1$ is given a positive answer in \cite[Theorem 5]{Cantor&G}.)\\
\indent Note that if $D=Q-Q$ for some finite set $Q\subset \Z^d$ then $\nu(D)\leq 1/|Q|$, with equality if $Q$ tiles $\Z^d$, that is if there exists $S\subset\Z^d$ such that $\Z^d=\bigsqcup_{r\in S} Q+r$. Thus, a question related to the one above is whether, given that a finite subset $Q\subset \Z^d$ tiles $\Z^d$, there must exist a periodic set $S$ such that $\Z^d=\bigsqcup_{r\in S} Q+r$. The \emph{periodic tiling  conjecture} posits an affirmative answer to the latter question (see \cite{L&W}).\\ \vspace{0.5cm}

\noindent \textbf{Acknowledgements.} The first named author was supported by the ERC Starting Grant Quasiperiodic and by the Balzan Research Project of J. Palis. The second named author was supported by the \'Ecole normale sup\'erieure, Paris, and by the Alfr\'ed R\'enyi Institute of Mathematics, and his work was part of project ANR-12-BS01-0011 CAESAR; he is also grateful to Gonzalo Fiz-Pontiveros and Bryna Kra for useful conversations.\\

\end{document}